\documentclass[onefignum,onetabnum]{siamltex1213}
 
\usepackage{tikz}
\usetikzlibrary{arrows}
\usetikzlibrary{patterns}
\usetikzlibrary{calc}
\usepackage[colorinlistoftodos,bordercolor=orange,backgroundcolor=orange!20,linecolor=orange,textsize=scriptsize]{todonotes}

\title{Combinatorial simplex algorithms can solve mean payoff games}
 
 \author{Xavier Allamigeon\footnotemark[2] \footnotemark[4]
\and Pascal Benchimol\footnotemark[2] \footnotemark[5]
\and St\'ephane Gaubert\footnotemark[2] \footnotemark[4]
\and Michael Joswig\footnotemark[3] \footnotemark[6] }

\usepackage[utf8]{inputenc}
\usepackage{frenchineq}
\usepackage{amsmath}  
\usepackage{amssymb}     
\usepackage{bbm}   
\usepackage{setspace}
   
\usepackage{bm}
\usepackage{hyperref}
\usepackage{accents}
\usepackage{mathrsfs}
\usepackage{mathtools} 
\usepackage{fixmath}

\newtheorem{assumption}{Assumption}

\newtheorem{primedassumption}{Assumption}

\newtheorem{remark}[theorem]{Remark}
\newtheorem{example}[theorem]{Example}

\DeclareSymbolFont{stmry}{U}{stmry}{m}{n}
\DeclareMathSymbol\mapsfromchar\mathrel{stmry}{"5B}

\newcommand{\ie}{\textit{i.e.},} 
 
\newcommand{\etc}{\textit{etc}}
\newcommand{\arxiv}[1]{arXiv:#1}

\DeclareMathOperator*{\sign}{sign}

\DeclareMathOperator*{\tper}{tper}

\DeclareMathOperator*{\Sym}{Sym}

\DeclareMathAlphabet\mathbfcal{OMS}{cmsy}{b}{n}

\DeclareMathAlphabet{\mathbbold}{U}{bbold}{m}{n}

\newcommand{\R}[0]{\mathbb{R}}

\newcommand{\K}[0]{K}

\newcommand{\trop}[1][]{
\ifthenelse{\equal{#1}{}}{ \mathbb{T} }{ \mathbb{T}(#1) }%
}    

\newcommand{\strop}[1][]{{\trop[#1]}_{\pm}}    
\newcommand{\postrop}[1][]{{\trop[#1]}_{+}}    
\newcommand{\negtrop}[1][]{{\trop[#1]}_{-}}

\newcommand{\transpose}[1]{#1^{\top}}

\newcommand{\tplus}{\oplus}          
\newcommand{\tsum}{\bigoplus}        
\newcommand{\ttimes}{\odot}       
\newcommand{\tdot}{\odot} 	
               
\newcommand{\tminus}{\ominus}

\newcommand{\zero}{\mathbbold{0}}    
\newcommand{\unit}{\mathbbold{1}}    
\newcommand{\tropP}{\mathcal{P}}

\newcommand{\puiseux}[0]{\mathbb{K}} 
\newcommand{\hahn}[1]{\R[\![t^{#1}]\!]}

\DeclareMathOperator*{\val}{val}

\DeclareMathOperator*{\sval}{sval}

\newcommand{\puiseuxP}[0]{\bm{\mathcal{P}}}

\newcommand{\x}[0]{\bm{x}}
\newcommand{\y}[0]{\bm{y}}

\newcommand{\p}[0]{\bm{p}}
\newcommand{\cc}[0]{\bm{c}}
\newcommand{\A}[0]{\bm{A}}

\newcommand{\W}[0]{\bm{W}}

\renewcommand{\b}[0]{\bm{b}}

\newcommand{\M}[0]{\bm{M}}

\newcommand{\puiseuxl}[0]{\bm{l}}
\newcommand{\puiseuxu}[0]{\bm{u}}

\newcommand{\ileaving}[0]{i^{\mathsf{out}}}
\newcommand{\ient}[0]{i^{\mathsf{ent}}}

\newcommand{\group}{G}

\newcommand{\groupPrime}{H}

\newcommand{\first}{\pi}

\newcommand{\tropLP}{\text{LP}}
\newcommand{\phaseI}{\text{Phase I}}
\newcommand{\puiseuxLP}{\bm{LP}}
\newcommand{\IpertLP}{\overline{\Ipert{LP}}}
\newcommand{\pertLP}{{\widetilde{\Ipert{LP}}}}

\newcommand{\germs}{\mathbb{G}}
\newcommand{\sgerms}{\mathbb{G}_{\!\pm}}

\newcommand{\perturb}[2][]{
\ifthenelse{\equal{#1}{}}{ \tilde{\Ipert{#2}} }{ #2[#1] }%
}
\newcommand{\pert}[2][]{
\ifthenelse{\equal{#1}{}}{ \widetilde{\Ipert{#2}} }{ \Ipert{#2}[#1] }%
}

\newcommand{\epsMat}{E}
\newcommand{\epsMatEnt}{\epsilon}

\newcommand{\Igerms}{\mathbb{I}}

\newcommand{\sIgerms}{\Igerms_{\!\pm}}

\newcommand{\second}{\rho}

\newcommand{\Ipert}[1]{
\mathsf{#1}}

\newcommand{\Iperturb}[1]{\Ipert{#1}}

\newcommand{\strategy}{\phi}
\newcommand{\strategytrop}{\strategy^{\text{trop}}}

\newcommand{\MinorOracle}{$\Omega$}

\begin{document}
\maketitle

\renewcommand{\thefootnote}{\fnsymbol{footnote}}
\footnotetext[2]{INRIA and CMAP, \'Ecole Polytechnique, 91128 Palaiseau Cedex France
\texttt{firstname.lastname@inria.fr}}
\footnotetext[3]{Institut f{\"u}r Mathematik,
 TU Berlin, 
 Str.\ des 17. Juni 136, 10623 Berlin, Germany
 \texttt{joswig@math.tu-berlin.de}}

\footnotetext[4]{X.~Allamigeon and S.~Gaubert are partially supported by the PGMO program of EDF and Fondation Math\'ematique Jacques Hadamard.}
\footnotetext[5]{P.~Benchimol is supported by a PhD fellowship of DGA and \'Ecole Polytechnique.}
\footnotetext[6]{M.~Joswig is partially supported by Einstein Foundation Berlin and DFG within the Priority Program 1489.}

\renewcommand{\thefootnote}{\arabic{footnote}}

\begin{abstract}
  A combinatorial simplex algorithm is an instance of the simplex method in which the pivoting depends on certain
  combinatorial data only. We show that any algorithm of this kind admits a tropical analogue which can be used to solve
  mean payoff games. Moreover, any combinatorial simplex algorithm with a strongly polynomial complexity (the existence
  of such an algorithm is open) would provide in this way a strongly polynomial algorithm solving mean payoff games.
  Mean payoff games are known to be in $\text{NP} \cap \text{co-NP}$; whether they can be solved in polynomial time is
  an open problem. Our algorithm relies on a tropical implementation of the simplex method over a real closed field of
  Hahn series. One of the key ingredients is a new scheme for symbolic perturbation which allows us to lift an arbitrary
  mean payoff game instance into a non-degenerate linear program over Hahn series.
\end{abstract}


\begin{AMS}\end{AMS}

\pagestyle{myheadings}
\thispagestyle{plain}

\section{Introduction}

The purpose of this paper is to establish a link between two notoriously open problems in theoretical computer science.
It is unknown whether or not there is a strongly polynomial algorithm for linear programming.  For mean payoff games,
even the existence of a (not necessarily strongly) polynomial time algorithm to decide which player has a winning
strategy is an open question.
Without offering a solution to either problem, we show that a solution for the first
problem, with additional properties, implies a solution for the second one.  Our proof uses tropical geometry, or rather
tropical linear algebra, in an essential way.

The question of the existence of a polynomial time algorithm for mean payoff games was originally raised
in~\cite{GKK-88}. Mean payoff games were shown to be in $\text{NP} \cap \text{co-NP}$ in~\cite{zwick}.  They were shown
to be equivalent to feasibility problems in tropical linear programming in~\cite{AGG}.  Classically, as well as
tropically, linear programming feasibility and linear optimization are equivalent.  In \cite{tropical+simplex} we
devised an algorithm to solve tropical linear optimization problems.  However, that algorithm is limited to primally and
dually non-degenerate problems.  Further, it only provides a Phase II simplex method, that is, it requires a tropical
basic point as additional input.  Classically, Phase I, which finds a first basic point, can be reduced to a Phase II
problem, but this requires to be able to deal with degenerate input.  Therefore the algorithm in \cite{tropical+simplex}
does \emph{not} solve arbitrary mean-payoff games.

For experts in classical linear programming it is tempting to underestimate the difficulty to generalize an algorithm
which works for non-degenerate tropical linear optimization problems to degenerate ones.  To explain this, it is useful
for a moment to attain an algebraic geometry perspective.  The tropicalization of an (affine) algebraic variety $V$,
that is, the joint vanishing locus of finitely many polynomials in $d$ indeterminates over a field with a
non-archimedian valuation, is a polytopal complex, the tropical variety $T(V)$, in $\R^d$, which is gotten by applying
the valuation map coordinate-wise to all points in $V$.  Key features of $V$ are visible in $T(V)$.  For instance, if
$C$ is an irreducible planar algebraic curve over an algebraically closed field, then $T(C)$ is a connected planar
graph.  In this case the genus of $C$ coincides with the rank of the first (co-)homology module of $T(V)$, that is, the
number of edges of $T(C)$ minus the number of edges of a spanning tree; see \cite{ItenbergMikhalkinShustin07}
\cite[\S1.7]{MaclaganSturmfels}.  This way the piece-wise linear object $T(C)$ encodes non-trivial exact information
about the arbitrarily non-linear object $C$.  In a similar fashion, our results can be interpreted to show that
mean-payoff games carry essential information about classical linear programs from a computational complexity point of
view.  It is a basic fact that the intersection of two algebraic varieties is again an algebraic variety, and it is
another basic fact that the tropical analogue does \emph{not} hold.  The impact for our algorithmic problems is the
following.  In a primally non-degenerate classical linear program in $d$ variables each basic point is the unique point
in the intersection of precisely $d$ affine hyperplanes, given by the constraints, and each edge of the feasible region
arises as the intersection of precisely $d-1$ affine hyperplanes.  In this situation tropicalization and intersection
commute, and the affine algebraic variety which is the intersection of the $d-1$ hyperplanes at an edge is a complete
intersection.  In the degenerate setting, however, this is no longer true and therefore the algorithm from
\cite{tropical+simplex}, which builds on the direct connection between the classical and the tropical worlds, cannot be
applied.

To overcome this obstacle, as one of our key contributions here, we introduce a new scheme for symbolic perturbation
which is tailored to the needs of tropical linear programming.  Our transformation of a combinatorial simplex algorithm
into a method to solve mean-payoff games then works by applying our algorithm from \cite{tropical+simplex} to perturbed
tropical linear programs.  In spirit, this perturbation scheme is similar to other techniques known in computational
geometry~\cite{Edelsbrunner87,EmirisCannySeidel97} and linear
optimization~\cite{Jeroslow1973,FilarAltmanAvrachenkov2002}.  Yet from the technical point of view our approach is quite
different as it requires to construct special fields of formal Hahn series.  We suspect that this has more applications
in tropical geometry, which are independent of algorithmic questions.

Another ingredient of our approach is the application of model theory arguments to the analysis of classical simplex
methods. This allows us to apply Tarski's principle for real-closed fields (to the fields of Hahn series required for
the perturbation) without losing the grip on exact complexity bounds.  Starting from a basic point the simplex method
traces a path in the vertex-edge graph which is directed by the linear objective function.  At each basic point on the
way, the pivoting rule decides which edge to pick next.  We call a pivoting rule \emph{combinatorial} if this decision
depends on evaluating signs of minors of the extended matrix only.  These signs, in particular, include the orientations
of edges of the feasible region.  For such pivoting rules we can show that the number of classical pivots yields an
upper bound for the number of their tropical counterparts.  This way we arrive at the main result of this paper.
Theorem~\ref{th:main} and Corollary~\ref{cor:main} say that a strongly polynomial classical simplex algorithm with a
combinatorial pivoting rule yields a strongly polynomial algorithm for solving mean payoff games, provided that each
pivoting step is polynomially bounded.

The essence of our method is to think of tropical polyhedra as images of classical polyhedra over ordered fields by a
non-archimedean valuation, in accordance with previous work in tropical convexity, in particular,
see~\cite{develin2004,DevelinSantosSturmfels05,DevelinYu07}.  Alternatively, we may consider images of classical
polyhedra over the field of real numbers by the ``valuation'' which takes the log of the absolute value. In this way,
tropical polyhedra can be obtained as limits of classical polyhedra defined either by inequalities or generators with
exponentially large coefficients. This approach, which is a variation on ``Viro's method''~\cite{viro2000} in tropical
geometry, or on ``Maslov's dequantization''~\cite{litvinov00}, has been applied by several authors to derive results of
tropical convexity, see in particular~\cite{BriecHorvath04}, where tropical convex sets are shown to be
Painlev\'e-Kuratowski limits of classical convex sets, and~\cite{GM08,AllamigeonGaubertKatzJCTA2011} for some
combinatorial applications.  A somehow related method was used in~\cite{schewe2009} to show that solving a mean payoff
game reduces to solving linear programs with exponentially large coefficients.  The size of the numerical data involved
in this reduction is not polynomially bounded in the length of the input, and so, along these lines, a strongly
polynomial linear programming algorithm would only yield an algorithm to solve mean payoff games that is
pseudo-polynomial in the Turing machine model.  The difficulty of working with exponentially large coefficients is
solved here by a symbolic (Hahn series) approach combined with perturbation arguments, at the price of limiting the
scope of the reduction to a specific class of pivoting algorithms.  We finally note that tropical linear programming
algorithms using different principles have been developed in~\cite{BA-08}, \cite[Chapter 10]{But:10}, and~\cite{GKS11}.

\section{Preliminaries}\label{sec:preliminaries}

\subsection{Tropical semirings} \label{subsec:trop_semirings}

A common example of tropical semiring is the max-plus semiring $(\R \cup \{-\infty\}, \max, +)$. In the present paper, more generally, we consider a semiring  $(\trop[\group], \tplus, \ttimes)$ based on a totally ordered abelian group $(\group, +, \leq)$, defined as follows. The base set is $\trop[\group] = \group \cup \{ \zero_{\trop[\group]} \}$ where the new element $\zero_{\trop[\group]} \not \in \group$ satisfies   $\zero_{\trop[\group]} \leq a$ for all $a \in \group$.
 The additive law $\tplus$ is defined by $a \tplus b = \max (a,b)$, where the maximum is taken with respect to $\leq$. The multiplication $\ttimes$ is the addition $+$ extended to $\trop[\group]$ by setting  $a + \zero_{\trop[\group]} = \zero_{\trop[\group]} + a = \zero_{\trop[\group]}$ for all $a \in \trop[\group]$.

The zero and unit elements are $\zero_{\trop[\group]}$ and $\unit_{\trop[\group]} := 0_\group$, the neutral element of $G$, respectively. 
In the following, we simply denote $\trop[\group]$ by $\trop$ when this is clear from the context.

The operations are extended to matrices $A = (A_{ij}), B = (B_{ij})$ with entries in $\trop$ by setting $A \tplus B = (A_{ij} \tplus B_{ij})$ and $A \tdot B = (\tsum_k A_{ik} \tplus B_{kj})$. In the following, unless explicitly stated, the entries of a matrix $A$ are denoted by $A_{ij}$. Moreover, we denote by $A_I$  the submatrix of $A$ obtained with rows indexed by $I$. By abuse of notation, we denote $A_{\{i\}}$ by $A_i$.

The \emph{signed tropical numbers} $\strop = \postrop \cup  \negtrop$ consist of two copies of $\trop $, the set of \emph{positive tropical numbers} $\postrop$ and the set of \emph{negative tropical numbers} $\negtrop$. These elements are respectively denoted as $a$ and $ \tminus a$ for $a \in \trop$. The elements $a$ and  $\tminus a$  are distinct unless $a = \zero_{\trop}$. In the latter case, these  two elements are identified, \ie\ we have $\zero_{\trop} = \tminus \zero_{\trop}$. 
The \emph{sign} of the elements  $ a$ and $\tminus a $ are $\sign(a) = 1$ and $\sign(\tminus a) = -1$, respectively,  when $a$ is not $\zero_{\trop}$, and $\sign(\zero_{\trop}) =0$.
The \emph{modulus} of $x \in \{ a, \tminus a \}$ is defined as $|x| := a$.
The \emph{positive part} $x^+$ of an element $x \in \strop$ is the tropical number $|x|$ if $x$ is positive, and it is $\zero_{\trop}$ otherwise.  The
negative part $x^-$ is similarly defined.  We have $x^+ \tplus x^-=|x|$.
Modulus, positive and negative parts extend to matrices entry-wise.
We also equip $\strop$ with a \emph{reflection map} $x \mapsto \tminus x$ which sends a positive element $a$ to $\tminus a$, and a negative element $\tminus a$ to $a$.

The tropical permanent of a square matrix $M \in \strop^{n \times n}$ is :
\begin{equation}  \label{eq:tdet}
\tper M := \tsum_{ \sigma \in \Sym([n])}  |M_{1 \sigma(1)}| \ttimes  \dots \ttimes |M_{n \sigma(n)}| = \max_{ \sigma \in \Sym([n])}   |M_{1 \sigma(1)}| +  \dots +|M_{n \sigma(n)}| \ ,
\end{equation}
where $\Sym([n])$ is the set of all permutations of $[n] := \{1, \dots, n \}$. We point out that $\tper M$ can be computed in $O(n^3)$ arithmetic operations by solving an optimal assignment problem. This also provides an optimal permutation $\sigma$. 

A matrix $M \in \strop^{n \times n}$ is \emph{tropically non-singular} if $\tper M \neq \zero_{\trop}$ and there is only one  permutation $\sigma$ attaining the maximum in the right-hand side of~\eqref{eq:tdet}.
In this case, we define the \emph{sign} of $\tper M$ to be the product $\sign(\sigma) \sign(M_{1 \sigma(1)}) \dots \sign(M_{n \sigma(n)})$. By extension, we say that the sign of $\tper M$ is 0 when $\tper M = \zero_{\trop}$.

An arbitrary matrix $M \in \strop^{m \times n}$ is \emph{tropically generic} if for every submatrix $W$ of $M$, either $\tper W = \zero_{\trop}$ or $W$ is tropically non-singular. Note that a tropically generic matrix may have $\zero_{\trop}$ entries. Also observe that, if $M$ is tropically generic, then the signs of the permanent of any square submatrix of $M$ are well-defined.  We call these the \emph{signs of the tropical minors} of $M$.

\subsection{Hahn series}
\label{sec:hahn-series}

Given a totally ordered abelian group $(\group, +, \leq)$, the field of \emph{Hahn series} $\hahn{\group}$ with value group $\group$ and with
real coefficient comprises the formal power series $\x := \sum_{ \alpha \in \Lambda } x_{\alpha} t^{\alpha}$ where the \emph{support} $\Lambda \subset
\group$ is well-ordered, and the coefficients $x_{\alpha}$ are non-zero real numbers. The \emph{valuation} $\val(\x)$ of a non-zero Hahn series $\x$
is $- \alpha_{\min}$, where $\alpha_{\min} = \min \{ \alpha \mid \alpha \in \Lambda \}$. By convention, we set $\val(0) = \zero_{\trop(\group)}$.
A non-zero Hahn series $\x$ is \emph{positive}, and we write $\x> 0$ in this case, when the leading coefficient $x_{-\val(\x)}$ is a positive real number. More generally, $\x \geq \y$ if $\x = \y$ or $\x - \y > 0$.

Throughout, we write $\puiseux$ instead of $\hahn{\group}$. Equipped with the usual operations on formal power series, $\puiseux$ forms a totally
ordered field.  The valuation map is an order-preserving homomorphism from the semiring $\puiseux_+$ of non-negative Hahn series to the tropical
semiring $\trop$.
More formally, for any series $\x,\y \in \puiseux$ satisfying $\x \geq 0$ and $\y \geq 0$, we have $\val(\x + \y) = \val(\x) \tplus \val(\y)$,
$\val(\x \y) = \val(\x) \ttimes \val(\y)$, and $\x \geq \y$ implies $\val(\x) \geq \val(\y)$. We extend the valuation map to vectors of $\puiseux^n$ coordinate-wise.

The \emph{signed valuation} $\sval$ is a map from $\puiseux$ to $\strop$ which sends $\x$ to $\val(x)$ if $\x$ is positive, and to $\tminus \val(\x)$ otherwise. Given $x \in \strop$, we denote by $\sval^{-1}(x)$ the set of all Hahn series $\x$ such that $\sval(\x) = x$. These two notations are extended to vectors and matrices entry-wise.

We call a matrix $\M \in \puiseux^{m \times n}$ \emph{generic} if for each square submatrix $\W \in \puiseux^{p \times p}$, either $\det \W \neq 0$,
or for all $\sigma \in \Sym([p])$, the term $(-1)^{\sign(\sigma)} \W_{1 \sigma(1)} \dots \W_{p \sigma(p)}$ is null.  Notice that our definition of
genericity is slightly more general than the more common requirement that all $p{\times}p$-minors are non-vanishing whenever $p\ge 2$.
Observe that if $M = \sval(\M)$ is tropically generic, then $\M$ is generic.  The converse does not hold.  If $M$ is tropically generic, then the signs of
the minors of $\M$ coincide with the signs of the tropical minors of $M$.

\subsection{From mean payoff games to tropical linear programming}
\label{sec:from-mean-payoff}

We recall the equivalence between mean payoff games
and tropical linear feasibility in the max-plus semiring $\trop[\R]$, where $\R$ is considered as the ordered group $(\R, +, \leq)$. In this setting, $\zero_{\trop[R]} = - \infty$. We refer the reader to~\cite{AGG} for details.
We shall describe a mean payoff game by a pair of {\em payment matrices}
$A ,B \in \trop[\R]^{m\times n}$. We also
fix an {\em initial state} $\bar{\jmath}\in[n]$.
The corresponding (perfect information)
game is played by alternating the moves of the two players,
called ``Max'' and ``Min'', 
as follows. We start from state $j_0:=\bar{\jmath}$.
Player Min
chooses a state $i_1\in [m]$ such that $A_{i_1 j_0} \neq - \infty$, and receives a
payment of $A_{i_1j_0}$ from
player Max. Then, Player Max chooses a state $j_1\in[n]$ such that  $B_{i_1j_1} \neq - \infty$ , and
receives a payment of $B_{i_1j_1}$ from Player Min. Player
Min again chooses a state $i_2\in [m]$ such that $A_{i_2 j_1} \neq -\infty$, receives a payment of $A_{i_2j_1}$ from
Player Max, and so on. If $j_0,i_1,j_1,i_2,j_2,\dots$ is the infinite sequence
of states visited in this way, the {\em mean
payoff} of player Max is defined to be
\[
\liminf_{p\to \infty} p^{-1}(-A_{i_1j_0}+B_{i_1j_1}-A_{i_2j_1}+B_{i_2j_2}+\cdots-A_{i_pj_{p-1}}+B_{i_pj_p}) \enspace  .
\]
It is assumed that $A$ has no identically $-\infty$ column, and that $B$ has no identically $-\infty$ row, so that Players Min and Max have at least one available action with finite payment, at each stage.
This game is known to have a value, which depends on the initial state $\bar{\jmath}$.
We denote it by $\chi_{\bar{\jmath}}$. We say that the initial state $\bar{\jmath}$ is winning
for Player Max if $\chi_{\bar{\jmath}}\geq 0$. The following theorem characterizes
the set of winning states. It is an immediate 
translation, in the language of linear programming, of a result of~\cite{AGG}.
\begin{theorem}[See Theorem~3.2 of~\cite{AGG}] \label{thm:MPG_to_tropLP}
The initial state $\bar{\jmath}\in[n]$ is winning for Player Max,
in the mean payoff game with payment matrices
$A, B $, if and only if there exists a solution $x\in \trop[\R]^n$
to the system of 
tropically linear inequalities 
\begin{equation*}
 x_{\bar{\jmath}} \geq 0 \ , \qquad A \ttimes x\leq B \ttimes x \ .
\end{equation*}
\end{theorem}

\begin{figure}
  \centering
  \begin{tikzpicture}[x=0.8cm, y=0.8cm,
 var/.style={circle,draw=black, thick,inner xsep=1ex, text width=2ex, align = center},
 hyp/.style={rectangle,draw=black, thick,inner xsep=1ex, text width=2ex, align=center, minimum size=3.5ex },
 edge/.style={draw=black,thick, >=stealth', ->} ]

\begin{scope}[shift={(0,0)}]

\node [var] (x) at (0,0) {$1$};
\node [var] (y) at (-2,2) {$2$};
\node [var] (z) at (0,4) {$3$};
\node [var] (t) at (2,2) {$4$};
\node [var] (u) at (4,2) {$5$};

\node [hyp] (1) at (0,2) {$1$};
\node [hyp] (2) at (-2,0) {$2$};
\node [hyp] (3) at (-2,4) {$3$};
\node [hyp] (4) at (2,4) {$4$};
\node [hyp] (5) at (2,0) {$5$};

\draw[edge] (x) --  (1);
\draw[edge] (1) -- node[anchor = south] {$-1$}(y);      
\draw[edge] (1) -- node[anchor = east] {$-2$}(z);      

\draw[edge] (y) -- (2);      
\draw[edge] (2) --  node[anchor = south] {$-3$} (x) ; 
\draw[edge] (2) to[out = 140, in = -120]  (-3,4) to[out= 60, in = 110] (z);             

\draw[edge] (z) -- (3);      
\draw[edge] (3) -- node[anchor = west] {$-4$} (y); 
\draw[edge] (3) to[out = 200, in = 120 ]  (-3,0) to[out= -60, in = 250] (x);        

\draw[edge] (z) -- (4);      
\draw[edge] (t) to[bend right] (4);      
\draw[edge] (4) to[bend right] node[anchor = east] {$+1$} (t);      
\draw[edge] (u) to[bend right] (4);

\draw[edge] (5) -- (x);
\draw[edge] (t) -- (5);
\draw[edge] (5) to[bend right]  node[anchor = south east] {$+2$}  (u);
 \end{scope}

  \end{tikzpicture}
\caption{A mean payoff game. The states in which Max plays are depicted by squares, while the states in which Min plays are depicted by circles. Edges represents valid moves, and are weighted by payments. An edge with no weight indicate a 0 payment.}
\label{fig:MPG}
\end{figure}
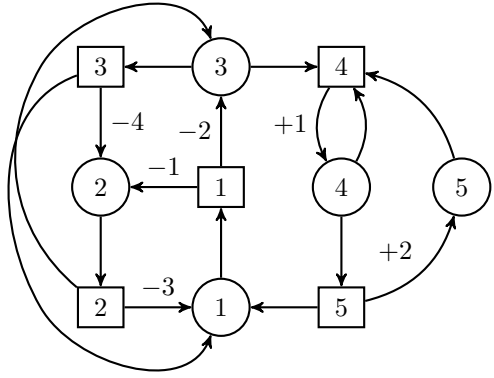

\begin{example}
 The mean payoff game with the following payment matrices is depicted in Figure~\ref{fig:MPG} (for the sake of readability, $-\infty $ entries are represented by the symbol ``$\cdot$''):
     \begin{equation*}
      A = 
      \begin{pmatrix}
         0 & \cdot  & \cdot  &\cdot &\cdot\\
         \cdot & 0 & \cdot & \cdot& \cdot\\
         \cdot & \cdot & 0 & \cdot& \cdot\\
         \cdot &\cdot & 0 & 0 &0 \\
          \cdot&\cdot&\cdot&0&\cdot
      \end{pmatrix}  , \quad
B= 
\begin{pmatrix}
   \cdot& -1 & -2 &\cdot &\cdot \\
  -3 & \cdot  &0 &\cdot &\cdot \\
   0 & -4 &\cdot &\cdot &\cdot \\
\cdot &\cdot &\cdot & +1 &\cdot \\
 0 &\cdot &\cdot &\cdot &+2
\end{pmatrix}
 \end{equation*}
In this game, the only winning initial states for Max are $\{4,5\}$. Indeed, the vector $(- \infty, - \infty, - \infty ,0, 0) $ is a solution of the following system of tropical linear inequalities, and any solution  $x \in \trop[\R]^5$ of this system satisfies $x_1=x_2=x_3 = -\infty$:
    \begin{align*}
      x_1 &\leq \max( x_2  -1,  x_3 -2) \\
      x_2 &\leq  \max( x_1-3,  x_3) \\
      x_3 &\leq  \max (x_1,x_2-4 ) \\
      \max(x_3, x_4, x_5) &\leq x_4 +1 \\
      x_4 &\leq \max(x_1, x_5+2) \ .
    \end{align*} 

\end{example}

\section{The simplex method for tropical linear programming}\label{sec:handl-degen-line}

A \emph{tropical polyhedron} is the solution set $x \in \trop^n$ of a finite number of tropically affine inequalities of the form
\begin{equation}
\alpha_1 \tdot x_1 \tplus \dots \tplus \alpha_n \tdot x_n \tplus \alpha_{n+1} \geq \beta_1 \tdot x_1 \tplus \dots \tplus \beta_n \tdot x_n \tplus \beta_{n+1} \label{eq:trop_affine_ineq}
\end{equation}
where the $\alpha_j$, $\beta_j$ belong $\trop$. Without loss of generality (see~\cite[Lemma~1]{GaubertKatz2011minimal}), it can be assumed that for all $j$, either $\alpha_j$ or $\beta_j$ is equal to $\zero_{\trop}$. For instance, the inequality $\max(1 + x_1, x_2) \geq \max(-1 + x_1, 3 + x_2)$ is obviously equivalent to $1 + x_1 \geq 3 + x_2$. As a consequence, any system of inequalities of the form~\eqref{eq:trop_affine_ineq} can be transformed (in linear time) into the form $A^+ \ttimes x \tplus b^+ \geq A^- \tdot x \tplus b^-$, where $A \in \strop^{m \times n} $ and $b \in \strop^m$ have tropically signed entries. 
We shall consider the following tropical polyhedron
\[\tropP(A,b):= \{ x \in \trop^n \mid A^+ \tdot x \tplus b^+ \geq A^- \tdot x \tplus b^- \}
\enspace .\]

A \emph{tropical linear program} $\tropLP(A, b , c)$ is an optimization problem of the form
\begin{equation}   \label{pb:tropLP} 
  \begin{array}{ll}
    \text{\rm  minimize} & \transpose{c} \tdot x \ \\
    \text{\rm subject to} & x \in \tropP(A,b) \ ,
  \end{array}
  \tag*{$\tropLP(A,b,c)$}
\end{equation}
where $A \in \strop^{ m \times n}$, $b \in \strop^{m}$, and $c \in \trop^n$.  We say that the program is \emph{infeasible} if the tropical polyhedron
$\tropP(A,b)$ is empty. Otherwise, it is said to be \emph{feasible}.

An important property underlying the previous work~\cite{tropical+simplex} is the connection between tropical linear programs and linear programs over Hahn series. Indeed, as $\puiseux$ is a totally ordered field, all the notions related to linear programming, in particular convex polyhedra, duality, \etc{}, naturally apply. Following this, given $\A \in \puiseux^{m \times n} $ and $\b \in \puiseux^m$, we denote $ \puiseuxP(\A,\b) \subset \puiseux^n$ the polyhedron defined by the system of inequalities $\A \x + \b \geq 0$ and $\x \geq 0$. 

\begin{proposition}[{\cite[Prop.~7]{tropical+simplex}}] \label{prop:lifting_linear_prog}
For any tropical linear program $\tropLP(A,b,c)$, there exists $\A \in \sval^{-1}(A)$, $\b \in \sval^{-1}(b)$ and $\cc \in \sval^{-1} (c)$ such that the linear program over $\puiseux$
 \begin{equation*}   \label{pb:puiseux_linear_prog_pb} 
   \begin{array}{ll}
     \text{\rm  minimize} & \transpose\cc  \x \ \\
     \text{\rm subject to} &  \x \in \puiseuxP(\A,\b) \,,
   \end{array} \tag*{$\puiseuxLP(\A,\b,\cc)$}
 \end{equation*}
satisfies the following properties.
\begin{itemize}
\item\label{item:lifting_feasible_set} The image under the valuation map of $\puiseuxP(\A,\b)$ is the tropical polyhedron $\tropP(A,b)$. In particular, $\puiseuxLP(\A,\b, \cc)$ is feasible if, and only if, $\tropLP(A,b,c)$ is feasible.
\item If $\puiseuxLP(\A,\b, \cc)$ admits $\x^*$ as an optimal solution,  then $\val(\x^*)$ is an optimal solution of $\tropLP(A,b,c)$.
\end{itemize}
\end{proposition}

The linear program $\puiseuxLP(\A, \b, \cc)$ in Proposition~\ref{prop:lifting_linear_prog} is called a \emph{lift} of $\tropLP(A, b, c)$. 

\subsection{Non-degenerate and bounded tropical linear programs}\label{subsec:generic_instances}

 The \emph{tropical simplex method} developed in~\cite{tropical+simplex} implicitly performs a simplex method on the Hahn linear program $\puiseuxLP(\A,\b,\cc)$, by doing only ``tropical'' computations over $\trop$, without explicitly manipulating Hahn series. This subsequently solves $\tropLP(A,b,c)$ by Proposition~\ref{prop:lifting_linear_prog}. However, this method applies only if the following assumptions are satisfied:
\begin{assumption}[Finiteness]\label{ass:finiteness}
The polyhedron $\tropP(A,b)$ 
 is bounded, and does not contain points with $\zero_{\trop}$ entries.
\end{assumption}

\begin{assumption}[Non-degeneracy] \label{ass:non_degeneracy}
  The matrix 
$\bigl(
\begin{smallmatrix}
A & b \\
\transpose{c} & \zero_{\trop} 
\end{smallmatrix}
\bigr)$ is tropically generic.
\end{assumption}
Under these assumptions, a \emph{tropical basic point} is defined as the unique point $x \in \tropP(A,b)$ activating a subset $I \subset [m]$ of $n$ inequalities in the system $A^+ \tdot x \tplus b^+ \geq A^- \tdot x \tplus b^-$ (\ie~for all $i \in I$, $A^+_i \tdot x \tplus b^+_i = A^-_i \tdot x \tplus b^-_i$), and such that $\tper A_I  \neq \zero_{\trop}$. The set $I$ is referred to as a \emph{tropical basis}. Tropical basic points are connected by \emph{tropical edges}, which are tropical line segments of the form $\{ x \in \tropP(A, b) \mid A^+_K \tdot x \tplus b^+_K = A^-_K \tdot x \tplus b^-_K \}$, where $K \subset [m]$ and $|K| = n-1$.

It turns out that tropical basic points and edges are intimately related to their analogues over Hahn series. Given $\A \in \sval^{-1}(A)$ and $\b \in \sval^{-1}(b)$, a basic point of $\puiseuxP(\A, \b)$ is characterized by a subset $I \subset [m]$ of active inequalities in the system $\A \x + \b \geq 0$, such that $\det \A_I \neq 0$.\footnote{Note that here the inequalities $\x_j \geq 0$ involved in the definition of $\puiseuxP(\A, \b)$ cannot be activated. Indeed, Assumption~\ref{ass:finiteness} ensures that $\puiseuxP(\A, \b)$ is contained in the open orthant $\x > 0$.} Then the tropical basic points are precisely the image under the valuation map of the basic points of $\puiseuxP(\A, \b)$. Similarly, it can be shown that tropical edges are the image of the edges of $\puiseuxP(\A, \b)$, and that the graph of incidence between basic points and edges in $\tropP(A, b)$ is the same as in $\puiseuxP(\A, \b)$. 

Starting from a given basic point, the algorithm of~\cite{tropical+simplex} visits the basic point/edge incidence graph of $\tropP(A, b)$, 
until it finds a basic point satisfying optimality conditions. We recall the main result of~\cite{tropical+simplex}:
\begin{theorem}[{\cite[Th.~1]{tropical+simplex}}]\label{th:tropical_simplex}
Under Assumptions~\ref{ass:finiteness} and~\ref{ass:non_degeneracy}, the tropical simplex algorithm  terminates and returns an optimal solution of $\tropLP(A,b,c)$ for any tropical pivoting rule. Every iteration (pivoting and computing reduced costs) can be done in $O(n(m+n))$ arithmetic operations over $\trop$, and in linear space. 
Moreover, the algorithm traces the image under the valuation map of the path followed by the classical simplex algorithm applied to any lift $\puiseuxLP(\A, \b, \cc)$
, with a compatible pivoting rule.
\end{theorem}

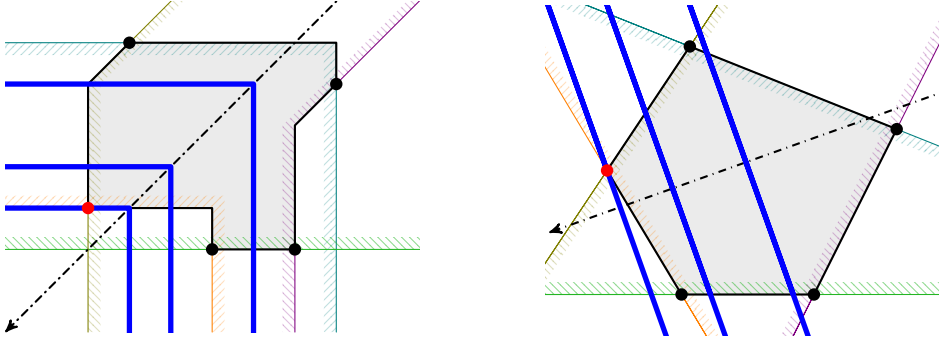
\begin{figure}
\centering
\begin{tikzpicture}[edge/.style={thick}, scale = 1]
     \colorlet{myblue}{rgb:red,0;green,102;blue,102}
      \colorlet{mygreen}{rgb:red,10;green,50;blue,10}
      \colorlet{purple}{rgb:red,50;green,0;blue,50}
      \colorlet{kaki}{rgb:red,50;green,50;blue,0}

      \colorlet{basicPoint}{black}
      \begin{scope} [shift = {(0,0)}, scale=0.55, line cap=round,line
        join=round,>=triangle 45,x=1.0cm,y=1.0cm]
         \clip(0,0) rectangle (10,8);
        
        \fill[fill=lightgray, fill opacity = 0.3] (2,3) -- (5,3) --
        (5,2) -- (7,2) -- (7,5) -- (8,6) -- (8,7) -- (3,7) -- (2,6) --
        cycle;

        \draw[ color=myblue] (0,7)-- (8,7) -- (8,0);
        \filldraw[draw=none,pattern=north east lines,pattern  color=myblue,fill opacity=0.5] 
        (0,7) -- (8,7) -- (8,0) -- (7.7,0) -- (7.7, 6.7) -- (0,6.7) --  cycle;
       
        \draw[color=orange] (0,3)-- (5,3)-- (5,0);
        \filldraw[draw=none,pattern=north east lines,pattern color=orange,fill opacity=0.5] 
        (0,3) -- (5,3) -- (5,0) --  (5.3, 0) -- (5.3, 3.3) -- (0,3.3) -- cycle;
      
        \draw[color=mygreen] (0,2)-- (11,2);
        \filldraw[draw=none,pattern=north west  lines,pattern color=mygreen,fill opacity=0.5]
        (0,2) -- (11,2) -- (11,2.3)  -- (0,2.3) --  cycle;

        \draw[color=purple] (7,0) -- (7,5)-- (12,10); 
        \filldraw[draw=none,pattern=north west lines,pattern color=purple,fill opacity=0.5] 
        (7,0) -- (7,5) -- (12,10) -- (11.7, 10) -- (6.7, 5) -- (6.7,0) -- cycle;
      
        \draw[color=kaki] (2,0) -- (2,6)-- (6,10); 
        \filldraw[draw=none,pattern=north west lines,pattern color=kaki,fill opacity=0.5] 
        (2,0) -- (2,6) -- (6,10) -- (6.3, 10) -- (2.3, 6) -- (2.3,0) -- cycle;
      
        \fill [basicPoint] (5,2) circle (0.15);
        \fill [basicPoint] (7,2) circle (0.15);
        \fill [basicPoint] (8,6) circle (0.15);
        \fill [basicPoint] (3,7) circle (0.15);

        \draw[edge] (2,3)--(5,3)--(5,2);
        \draw[edge] (5,2)--(7,2);
        \draw[edge] (7,2)--(7,5)--(8,6);
        \draw[edge] (8,6)--(8,7)--(3,7);
        \draw[edge] (3,7)--(2,6)--(2,3);

         \draw[dashdotted, thick,<-, >=stealth'] (0,0) -- (10,10);
         \draw[color = blue, line width = 2] (-1,6) -- (6,6) -- (6,-1);
        \draw[color = blue, line width = 2] (-1,4) -- (4,4) --  (4,-1);
        \draw[color = blue, line width = 2] (-1,3) -- (3,3) --(3,-1);
        
        \fill [red] (2,3) circle (0.15);

      \end{scope}

  \begin{scope} [ shift = {(8,0.5)}, scale=0.55, x=1.0cm,y=1.0cm, 
      extended line/.style={shorten >=-10cm,shorten <=-10cm} ]
      \clip (-1.5,-1) rectangle (8,7);
     
     \coordinate (orange_green) at (1.8,0);
     \coordinate (green_purple) at (5,0);
     \coordinate (purple_blue) at (7,4);
     \coordinate (orange_kaki) at (0,3);
     \coordinate (kaki_blue) at (2,6);

       \fill[fill=lightgray, fill opacity = 0.3]
       (orange_green) -- (green_purple) -- (purple_blue) -- (kaki_blue) -- (orange_kaki) -- cycle;

     \coordinate (blue_dir) at ($ (kaki_blue) - (purple_blue)$);
     \coordinate (orange_dir) at ($ (orange_green) - (orange_kaki)$);
     \coordinate (green_dir) at ($ (orange_green) - (green_purple)$);
     \coordinate (purple_dir) at ($ (green_purple) - (purple_blue)$);
     \coordinate (kaki_dir) at ($ (orange_kaki) - (kaki_blue)$);

     \draw[color=myblue] (kaki_blue) -- + ($10*(blue_dir)$) -- + ($-10*(blue_dir)$);
     \filldraw [draw=none,pattern=north east lines,pattern color=myblue,fill opacity=0.5] 
      ($(kaki_blue) + 10*(blue_dir)$) -- ++(0,-0.3)-- ++ ($-20*(blue_dir)$) -- ++(0,0.3) -- cycle; 

      \draw[color=orange]  (orange_green) -- +  ($10*(orange_dir)$) -- + ($-10*(orange_dir)$);
       \filldraw[draw=none,pattern=north east lines,pattern color=orange,fill opacity=0.5] 
      ($(orange_green) + 10*(orange_dir)$) -- ++(0.3,0)-- ++ ($-20*(orange_dir)$) -- ++(-0.3,0) -- cycle; 

      \draw[color=mygreen] (orange_green) -- +  ($10*(green_dir)$) -- + ($-10*(green_dir)$);
       \filldraw[draw=none,pattern=north west lines,pattern color=mygreen,fill opacity=0.5] 
      ($(orange_green) + 10*(green_dir)$) -- ++(0,0.3)-- ++ ($-20*(green_dir)$) -- ++(0,-0.3) -- cycle; 

      \draw[color=purple]  (green_purple) -- +  ($10*(purple_dir)$) -- + ($-10*(purple_dir)$);;
       \filldraw[draw=none,pattern=north west lines,pattern color=purple,fill opacity=0.5] 
      ($(green_purple) + 10*(purple_dir)$) -- ++(-0.3,0)-- ++ ($-20*(purple_dir)$) -- ++(0.3,0) -- cycle; 

      \draw[color=kaki, extended line]  (orange_kaki) --  +  ($10*(kaki_dir)$) -- + ($-10*(kaki_dir)$);
      \filldraw[draw=none,pattern=north west lines,pattern color=kaki,fill opacity=0.5] 
      ($(orange_kaki) + 10*(kaki_dir)$) -- ++(0.3,0)-- ++ ($-20*(kaki_dir)$) -- ++(-0.3,0) -- cycle; 

      \draw[edge] (orange_kaki) -- (orange_green);
      \draw[edge] (orange_green) -- (green_purple);
      \draw[edge] (green_purple) -- (purple_blue);
      \draw[edge] (purple_blue) -- (kaki_blue);
      \draw[edge] (kaki_blue) -- (orange_kaki);

      \fill [basicPoint] (orange_green) circle (0.15);
      \fill [basicPoint] (green_purple) circle (0.15);
      \fill [basicPoint] (purple_blue) circle (0.15);
      \fill [basicPoint] (kaki_blue) circle (0.15);
     
      \coordinate (objective) at ($(7,2.5)$);
      \path let \p1=(objective) in
      coordinate (level_dir) at ($10*(-\y1, \x1)$);
      \draw[dashdotted, thick,<-, >=stealth'] ($(0,2) -2/10*(objective)$) -- ($(0,2)+10*(objective)$) ; 
      \draw[color = blue, line width = 2] (4.5,0) -- + (level_dir) -- + ($-1*(level_dir)$);
      \draw[color = blue, line width = 2] (2.5,0) -- + (level_dir) -- + ($-1*(level_dir)$);
      \draw[color = blue, line width = 2] (0,3) -- + (level_dir) -- + ($-1*(level_dir)$);
     
      \fill[red] (orange_kaki) circle(0.15cm);

\end{scope}

\end{tikzpicture}
\caption{A tropical linear program (on the left) and a lift of this program to Hahn series (on the right). Dotted lines represent objective functions and blue lines are level sets. Optimal basic points are red dots. Other basic points, and the edges, are depicted in black.}
\end{figure}

In the classical setting, reduced costs are used to determine which edge can be selected in order to improve the objective function. The reduced cost vector of a basis $I$ is the unique solution $\y^I$ of the system $\transpose{\A} \y +  \cc = 0$,  and $\y_i = 0 $ for all $i \in [m] \setminus I$. 
The entries of $\y^I$ are given by Cramer determinants of the latter system, and so they can be expressed in terms of some minors of the matrix $( \transpose{\A_I} \ \cc)$.

 Thanks to Assumption~\ref{ass:non_degeneracy}, their signed valuation can be determined in a tropical way, 
by computing the tropical analogue of Cramer determinants, see~\cite{tropical+simplex} for details.
This provides the tropical reduced cost vector. As a consequence, if the $i$-th entry of the latter vector is tropically negative (\ie~in $\trop_-$), pivoting along the edge of $\puiseuxP(\A, \b)$ associated with the set $K = I \setminus \{i\}$ decreases the objective function $\x \mapsto \transpose{\cc} \x$. 

\begin{remark}\label{remark:max}
The tropical simplex method can be used to maximize $x \mapsto \transpose{c} \tdot x$, as $\tropP(A, b)$ is bounded by Assumption~\ref{ass:finiteness}. It only suffices to select a leaving variable $i$ such that the $i$-th entry of reduced cost vector is positive, \ie\ to reverse the minimization.  All the conclusions of Theorem~\ref{th:tropical_simplex} also apply to the maximization problem. 
\end{remark}

\subsection{General tropical linear programs}\label{subsec:general_algorithm}

We now explain how to handle any tropical linear program $\tropLP(A,b,c)$ over $\trop[\R]$ ($\tropLP$ for short) when Assumptions~\ref{ass:finiteness} and~\ref{ass:non_degeneracy} are not necessarily satisfied. 

The optimal value of the previous problem is necessarily attained on a (tropically) extreme point of $\tropP(A,b)$. Indeed, it is known that a tropical
polyhedron admits an analogue of the Minkowski-Weyl description
by extreme points and rays~\cite{GK06}. The fact a linear form
achieves its minimum at an extreme point readily follows
from this result.
The entries of extreme points can be bounded as follows.
\begin{lemma}[{\cite[Prop.~10]{AllamigeonGaubertKatzLAA2011}}]
Let $x$ be an extreme point of $\tropP(A,b)$. Then, for all $j \in [n]$, we have $x_j \leq u$, where $u := 2n \max(\max_{ij} |A_{ij}|, \max_i |b_i|)$.
\end{lemma}

As a consequence, the solution of $\tropLP(A,b,c)$ does not change if we explicitly add the constraints $x_j \leq u$. 

We may think that lower bound constraints could be added as well in order to satisfy Assumption~\ref{ass:finiteness}. Indeed, it can be shown that if $x$ is extreme and $x_j \neq \zero_{\trop}$, then $x_j \geq -u$. However, adding the latter constraints does not provide an equivalent problem. More precisely, the initial problem may be feasible, but $\tropP(A,b)$ may contain only points satisfying $x_j = \zero_{\trop}$. In this case, adding the constraint $x_j \geq -u$ would provide an infeasible problem.

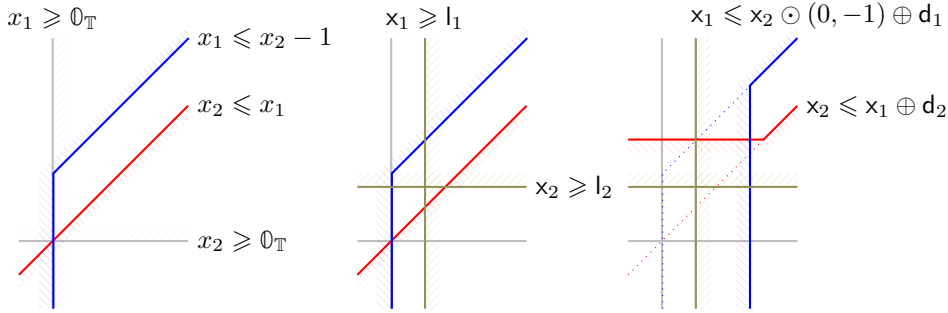
\begin{figure}
  \centering
  \begin{tikzpicture}[scale=0.9]
  
    \begin{scope}[shift={(0,0)}]
      \draw[gray!50!, thick] (-0.5,-0) -- (2,-0);
      \filldraw[draw=none,pattern=north east lines,pattern color= gray!50!,fill opacity=0.3] (-0.5,0) -- (2,0) -- (2,0.2) -- (-0.5,0.2) -- cycle;
      
      \draw[gray!50!, thick] (0,-1) -- (0,3);
      \filldraw[draw=none,pattern=north east lines,pattern color=gray!50!,fill opacity=0.3] (0,-1) -- (0,3) -- (0.2,3) -- (0.2,-1) -- cycle;

      \draw[red, thick] (-0.5,-0.5) -- (2,2);
      \filldraw[draw=none,pattern=north west lines,pattern color=red,fill opacity=0.3] (-0.5,-0.5) -- (2,2) -- (2,1.8) -- (-0.5,-0.7) --cycle;

      \draw[blue, thick] (0.01,-1) -- (0.01,1) -- (2.01,3);
      \filldraw[draw=none,pattern=north west lines,pattern color=blue,fill opacity=0.3] (0.01,-1) -- (0.01,1) -- (2.01,3) -- (2.01,3.2) -- (-0.21,1) -- (-0.21,-1)--cycle;

      \node[anchor = south] at (0,3) {$ x_1 \geq \zero_{\trop}$};
      \node[anchor = west] at (2,0) {$ x_2 \geq \zero_{\trop}$};

      \node[anchor = west] at (2,3) {$ x_1 \leq x_2 -1 $};
      \node[anchor = west] at (2,2) {$ x_2 \leq x_1  $};

    \end{scope}

    \begin{scope}[shift={(5,0)}]
      \draw[gray!50!, thick] (-0.5,-0) -- (2,-0);
      
      \draw[gray!50!, thick] (0,-1) -- (0,3);

      \draw[red, thick] (-0.5,-0.5) -- (2,2);
      \filldraw[draw=none,pattern=north west lines,pattern color=red,fill opacity=0.3] (-0.5,-0.5) -- (2,2) -- (2,1.8) -- (-0.5,-0.7) --cycle;

      \draw[blue, thick] (0.01,-1) -- (0.01,1) -- (2.01,3);
      \filldraw[draw=none,pattern=north west lines,pattern color=blue,fill opacity=0.3] (0.01,-1) -- (0.01,1) -- (2.01,3) -- (2.01,3.2) -- (-0.21,1) -- (-0.21,-1)--cycle;

        \draw[yellow!50!black, thick] (-0.5,0.8) -- (2,0.8);
      \filldraw[draw=none,pattern=north east lines,pattern color= yellow!50!black,fill opacity=0.3] (-0.5,0.8) -- (2,0.8) -- (2,1) -- (-0.5,1) -- cycle;
      
      \draw[yellow!50!black, thick] (0.5,-1) -- (0.5,3);
      \filldraw[draw=none,pattern=north east lines,pattern color=yellow!50!black,fill opacity=0.3] (0.5,-1) -- (0.5,3) -- (0.7,3) -- (0.7,-1) -- cycle;

      \node[anchor = south] at (0.5,3) {$ \Ipert{x}_1 \geq \Ipert{l}_1$};
      \node[anchor = west] at (2,0.8) {$ \Ipert{x}_2 \geq \Ipert{l}_2$};
      
    \end{scope}

    \begin{scope}[shift={(9,0)}]
      \draw[gray!50!, thick] (-0.5,-0) -- (2,-0);
      
      \draw[gray!50!, thick] (0,-1) -- (0,3);

      \draw[red, dotted] (-0.5,-0.5) -- (2,2);
      \draw[red, thick] (-0.5,1.5) -- (1.5,1.5) -- (2,2);
      \filldraw[draw=none,pattern=north west lines,pattern color=red,fill opacity=0.3] (-0.5,1.5) -- (1.5,1.5) -- (2,2) -- (2,1.8) -- (1.5, 1.3) -- (-0.5,1.3) --cycle;

      \draw[blue, dotted] (0.01,-1) -- (0.01,1) -- (2.01,3);
      \draw[blue, thick] (1.3,-1) -- (1.3,2.3) -- (2,3);
      \filldraw[draw=none,pattern=north west lines,pattern color=blue,fill opacity=0.3] (1.3,-1) -- (1.3,2.3) -- (2,3) -- (2,3.2) -- (1.1,2.3) -- (1.1,-1)--cycle;

        \draw[yellow!50!black, thick] (-0.5,0.8) -- (2,0.8);
      \filldraw[draw=none,pattern=north east lines,pattern color= yellow!50!black,fill opacity=0.3] (-0.5,0.8) -- (2,0.8) -- (2,1) -- (-0.5,1) -- cycle;
      
      \draw[yellow!50!black, thick] (0.5,-1) -- (0.5,3);
      \filldraw[draw=none,pattern=north east lines,pattern color=yellow!50!black,fill opacity=0.3] (0.5,-1) -- (0.5,3) -- (0.7,3) -- (0.7,-1) -- cycle;

       \node[anchor = south] at (2.3,3) {$ \Ipert{x}_1 \leq \Ipert{x}_2 \ttimes (0, -1) \tplus \Ipert{d}_1 $};
       \node[anchor = west] at (2,2) {$ \Ipert{x}_2 \leq \Ipert{x}_1 \tplus \Ipert{d}_2  $};

    \end{scope}

  \end{tikzpicture}
  \caption{Illustration of a tropical polyhedron with $\zero_{\trop}$ entries (left).  Two perturbation steps (middle and right); see
    Example~\ref{exmp:perturbation}.}
\label{fig:perturbation}
\end{figure}

To overcome this difficulty, we propose to lift our initial problem into a richer semiring, where we can replace the $\zero_{\trop}$ coefficients by finite entries. This semiring $\Igerms$ is defined as $\trop[\R^2]$, where the group $\R^2$ is endowed with entry-wise addition and lexicographical order. Intuitively, a pair $(\alpha, \beta)$ corresponds to a scalar $\alpha M + \beta$, where $M$ is an infinite formal value. 
Finite elements $\beta$ of $\trop$ are encoded as scalars of the form $(0, \beta)$. In contrast, the elements of $\Igerms$ of the form $(\alpha, \cdot)$ with $\alpha \neq 0$ correspond to different layers of infinite values, namely $-\infty$ if $\alpha < 0$, and $+\infty$ if $\alpha > 0$. Finally, the semiring $\Igerms$ has its own bottom element, denoted by $\zero_{\Igerms}$.
Following this interpretation,  we say that $x=(\alpha, \cdot)$ is \emph{infinitely smaller} than $y =(\alpha', \cdot)$ if $\alpha < \alpha'$. In this case, we write $x \ll y$. By extension, we set $ \zero_{\Igerms} \ll (\alpha, \cdot)$ for any $\alpha \in \R$.

We now lift the matrix $A \in \strop^{m \times n}$ into $\Ipert{A}  \in \sIgerms^{m \times n}$, defined as follows:
\begin{equation*}
  \Iperturb{A}_{ij} = 
  \begin{cases}
    (0, |A_{ij}|) & \text{if } A_{ij} \text{ is tropically positive} \ ,\\
\tminus  (0, |A_{ij}|) & \text{if } A_{ij} \text{ is tropically negative} \ ,\\
\zero_{\Igerms} & \text{if } A_{ij} = \zero_{\trop} \ .
  \end{cases}
\end{equation*}
The vectors $\Ipert{b} \in \sIgerms^m$, $\Ipert{c} \in \Igerms^n $ and the scalar $\Ipert{u} \in \Igerms$ are built from $b$, $c$ and $u$ similarly. 

The idea is to consider the polyhedron $\tropP(\Ipert{A}, \Ipert{b})$ with the additional constraints $\Ipert{u} \geq \Ipert{x}_j \geq \Ipert{l}_j$, where the $\Ipert{l}_j$ are ``infinitely small'' but finite entries, \ie\ of the form $(\alpha_j, \cdot)$ with $\alpha_j < 0$. 
We want to lift any $x \in \tropP(A,b)$ to an element $\Ipert{x}$ defined by $\Ipert{x}_j = (0, x_j)$ if $x_j \neq \zero_{\trop}$,  and $\Ipert{x}_j =l_j$ otherwise.  However, such a lift $\Ipert{x}$ may not satisfy the inequality $\Ipert{A}^+_i \tdot \Ipert{x} \tplus \Ipert{b}^+_i \geq \Ipert{A}^-_i \tdot \Ipert{x} \tplus \Ipert{b}^-_i$ when $\Ipert{b}_i = \zero_{\Igerms}$.
We circumvent this difficulty by perturbing  $\Ipert{b}^+$ into $\Ipert{b}^+ \tplus \Ipert{d}$, where $\Ipert{d} \in \Igerms^m$ is chosen such that 
$\max_j \Ipert{l}_j \ll \Ipert{d}_i \ll \unit_{\Igerms}$ for all $i \in [m]$ (in particular, $\Ipert{d}_i \neq \zero_\Igerms$).
\begin{example}\label{exmp:perturbation}
  Consider the tropical polyhedron $\tropP$ given by the 
  tropical linear inequalities 
  $x_1\le x_2-1$ 
  and $x_2\le x_1$.  These constraints violate Assumption~\ref{ass:finiteness}, and
  $\tropP$ consists of the single point $(\zero_{\trop},\zero_{\trop})$; see Figure~\ref{fig:perturbation} (left).  Cast into $\Igerms$, and
  with the additional constraints $\Ipert{x}_j \geq \Ipert{l}_j$, the resulting tropical polyhedron is empty; see
  Figure~\ref{fig:perturbation} (middle).  With the additional affine perturbation $\Ipert{d}$ we obtain the tropical
  polyhedron defined by the  tropical linear inequalities $ \Ipert{x}_1 \leq \Ipert{x}_2 \ttimes (0, -1) \tplus
  \Ipert{d}_1 $ and $ \Ipert{x}_2 \leq \Ipert{x}_1 \tplus \Ipert{d}_2 $ along with $\Ipert{x}_j \geq \Ipert{l}_j$; see Figure~\ref{fig:perturbation} (right).
\end{example}

 To summarize, the problem we are considering now is the following:
\begin{equation}   \label{pb:IpertLP} 
\begin{array}{ll}
\text{\rm  minimize} & \, \transpose{\Ipert{c}} \tdot \Ipert{x} \ \\
\text{\rm subject to} &
\begin{aligned}[t]
\Ipert{A}^+ \tdot \Ipert{x} \tplus (\Ipert{b}^+ \tplus \Ipert{d}) & \geq \Ipert{A}^- \tdot \Ipert{x} \tplus \Ipert{b}^- \\
\Ipert{u} & \geq \Ipert{e} \tdot \Ipert{x} \\
\Ipert{x} & \geq \Ipert{l} \ ,
\end{aligned}
\end{array}
\tag{$\IpertLP$}
\end{equation}
where $\Ipert{e}$ is the row vector of size $n$ with all entries equals to $\unit_{\Igerms}$.
Because of the constraints $\Ipert{u} \geq \Ipert{x}_j \geq \Ipert{l}_j$,  the feasible points of $\IpertLP$ have entries  of the form $(\alpha, \beta)$ with $\alpha \leq 0$. We project these elements to $\trop$ with the map $\second$, defined  by 
$\second(0, \beta) = \beta$, and $ \second(\alpha, \beta) =  \zero_{\trop}$  for $\alpha < 0$. The map $\second$ is extended to vectors entry-wise.
\begin{proposition}\label{prop:IpertLP}
The image under $\second$ of the feasible set of $\IpertLP$ is precisely the  feasible set of $\tropLP(A,b,c)$.
Moreover, if $\Ipert{x}$ is an optimal solution of $\IpertLP$, then $\second(\Ipert{x})$ is an optimal solution of $\tropLP(A,b,c)$.
\end{proposition}
\begin{proof}
Let us show that for any $x \in \tropP(A,b)$,  the feasible set of $\IpertLP$ contains the lift  $\Ipert{x}$ defined by $\Ipert{x}_j = (0, x_j)$ if $x_j \neq \zero_{\trop}$,  and $\Ipert{x}_j =\Ipert{l}_j$ otherwise.
 This point clearly satisfies the constraint $\Ipert{A}^+_i \tdot \Ipert{x} \tplus \Ipert{b}^+_i  \tplus \Ipert{d}^+_i \geq \Ipert{A}^-_i \tdot \Ipert{x} \tplus \Ipert{b}^-_i$ if $b_i \neq \zero_{\trop}$. Indeed, the inequality $A_i^+ \tdot x \tplus b^+_i \geq A_i^- \tdot x \tplus b^-_i$ ensures that $A_i^+ \tdot x \tplus b^+_i \neq \zero_{\trop}$, in which case we have: 
\[
\Ipert{A}^+_i \tdot \Ipert{x} \tplus \Ipert{b}^+_i = (0, A_i^+ \tdot x \tplus b^+_i) \ .
\] 
Besides, it can be verified that $\Ipert{A}^-_i \tdot \Ipert{x} \tplus \Ipert{b}^-_i$ is either equal to $(0, A^-_i \tdot x \tplus b^-_i)$ if $A^-_i \tdot x \tplus b^-_i \neq \zero_{\germs}$, and otherwise it is  of the form $(\alpha, \cdot)$  with $\alpha < 0$, or equal to $\zero_{\Igerms}$

It remains to consider the case $b_i = \zero_{\trop}$. If $A^+_i \tdot x \tplus b^+_i \neq \zero_{\trop}$, then the arguments above are still valid. Otherwise,  $ A_i^- \tdot x \tplus b^-_i = \zero_{\trop}$ and thus  $\Ipert{A}^-_i \tdot \Ipert{x} \tplus \Ipert{b}^-_i$ is either $\zero_{\Igerms}$ or of the form $\tsum_{j} (0, A_{ij}) \ttimes \Ipert{l}_j$. Since $\Ipert{d}_i \gg \Ipert{l}_j$, we deduce that $\Ipert{d}_i \geq \Ipert{A}^-_i \tdot \Ipert{x} \tplus \Ipert{b}^-_i$.

We next show that conversely, the image by $\rho$ of every feasible point
for $\IpertLP$ is a feasible point for $\tropLP(A,b,c)$. To see
this, let $\Igerms^{\leq}$ denote the subset of
$\Igerms$ consisting of the elements $(\alpha,\beta)$
with $\alpha\leq 0$, together with $\zero_{\Igerms}$. We observe
that $\Igerms^{\leq}$ is a subsemiring of $\Igerms$, and that
all the coefficients involved in $\IpertLP$, as well
as the entries of any feasible point of this program,
belongs to $\Igerms^{\leq}$. The announced
property and the rest of the proposition
follow from the fact that $\rho$ is a homomorphism
of semiring from $\Igerms^{\leq} $ to $\trop$,
noting that this implies in particular that $\rho$
is order preserving.
\end{proof}

Problem $\IpertLP$ satisfies Assumption~\ref{ass:finiteness}. It now remains to deal with Assumption~\ref{ass:non_degeneracy}. We propose to handle it by, again, embedding $\Igerms = \trop[\R^2]$ into a richer semiring.   Consider $ \germs :=\trop[\R^2 \times \groupPrime]$, where $\groupPrime$ is an abelian totally ordered group. The elements of $\groupPrime$  encode infinitesimal symbolic perturbations. We suppose that $\R^2 \times \groupPrime$ is ordered lexicographically, and that it is equipped with the coordinate-wise addition. 

Given $\Ipert{M}  \in \sIgerms^{ p \times q}$ and $\epsMat = (\epsMatEnt_{ij}) \in \groupPrime^{p \times q}$, we define the matrix $\pert[\epsMat]{M} = (\pert{M}_{ij})$ of size $p \times q$ with entries in $\sgerms$ as follows:
\begin{align*}
  \pert{M}_{ij} = 
  \begin{cases}
    (|\Ipert{M}_{ij}|, \epsMatEnt_{ij}) & \text{if} \  \Ipert{M}_{ij} \ \text{is tropically positive} \\
\tminus (|\Ipert{M}_{ij}|, - \epsMatEnt_{ij}) & \text{if} \ \Ipert{M}_{ij} \ \text{is tropically negative} \\
    \zero_{\germs} &  \text{if} \ \Ipert{M}_{ij} = \zero_{\Igerms}
  \end{cases}
\end{align*}
 The matrix $\epsMat$ is said to be \emph{sufficiently generic} if $\pert[\epsMat]{M}$ is tropically generic for all matrix $\Ipert{M}$. The following lemma provides an example of such a 
 matrix $\epsMat$.
\begin{lemma}\label{lemma:perturbation}
Let $p, q \geq 1$, and instantiate $\groupPrime$ by $\mathbb{R}^{q}$. Consider the matrix $\epsMat$ whose $(i,j)$-th entry is the vector $i \delta^j$, where $\delta^j$ is the $j$-th element of the canonical basis of $\mathbb{R}^{q}$. 
Then $\epsMat$ is sufficiently generic.
\end{lemma}
\begin{proof}
Consider a submatrix $\pert{M}'$ of $\pert[\epsMat]{M}$ with $\tper \pert{M}' \neq \zero_{\germs}$. If $\sigma$ and $\pi$ are two bijections attaining the maximum in $\tper \pert{M}'$, then $\sum_{j} \pm \sigma^{-1}(j) \delta^j =  \sum_{j} \pm \pi^{-1}(j) \delta^j$. The latter  vector equality  holds  if and only if $\sigma = \pi$.
\end{proof}

In the following, we suppose that $\groupPrime = \R^{n+1}$, and $\epsMat \in \groupPrime^{(m+n+2) \times (n+1)}$ is the matrix described in Lemma~\ref{lemma:perturbation}. We consider the problem:
\begin{equation*}   
\begin{array}{ll}
\text{\rm  minimize} & \, \transpose{\pert{c}} \tdot \pert{x} \ \\
\text{\rm subject to} &
\begin{aligned}[t]
\pert{A}^+ \tdot \pert{x} \tplus (\pert{b}^+ \tplus \pert{d}) & \geq \pert{A}^- \tdot \pert{x} \tplus \pert{b}^- \\
\pert{u} & \geq \pert{e} \tdot \pert{x} \\
\pert{Id}_n \tdot \pert{x} & \geq \pert{l} \ ,
\end{aligned}
\end{array}
\tag{$\pertLP$}\label{pb:pertLP} 
\end{equation*}
where we have set:
\[
\begin{pmatrix} 
\pert{A} & \pert{b} \tplus \pert{d} \\
\tminus \pert{e} & \pert{u} \\
\pert{Id}_n & \tminus \pert{l} \\
\transpose{\pert{c}} & \zero_{\germs} 
\end{pmatrix}
:=  
\pert[\epsMat]{\begin{pmatrix} 
\Ipert{A} & \Ipert{b} \tplus \Ipert{d} \\
\tminus \Ipert{e} & \Ipert{u} \\
\Ipert{Id}_n & \tminus \Ipert{l} \\
\transpose{\Ipert{c}} & \zero_{\Igerms} 
\end{pmatrix}}
\ . \footnotemark
\]
\footnotetext{By abuse of notation, $\Ipert{b} \tplus \Ipert{d}$ is the vector whose $i$-th entry is $\Ipert{b}_i$ if $\Ipert{b}_i \neq \zero_{\Igerms}$ and $\Ipert{d}_i$ otherwise.}

By construction, $\pertLP$ satisfies Assumptions~\ref{ass:finiteness} and~\ref{ass:non_degeneracy}. Moreover, the feasible points of~$\pertLP$ have entries of the form $(\Ipert{x}, \epsMatEnt)$, where $\Ipert{x} \in \R^2$ and $\epsMatEnt \in \groupPrime$.  Consider the map $\pi: \R^2 \times \groupPrime \mapsto \R^2$ defined by $\first(\Ipert{x}, \epsMatEnt ) = \Ipert{x}$. This map is extended to vectors entry-wise.

\begin{proposition}\label{prop:germLP}
Assume that the entries of $\epsMat$ are greater than $0_\groupPrime$. Then Problem~$\pertLP$ is feasible if, and only if, $\tropLP(A,b,c)$ is feasible. 
Besides, if $\pert{x} $ is an optimal solution of~$\pertLP$, then
$\second(\first(\pert{x}))$ is an optimal solution of $\tropLP(A,b,c)$.
\end{proposition} 
\begin{proof}
Consider a feasible point $\Ipert{x} \in \Igerms^n$ of $\IpertLP$. It can be lifted to  $\pert{x}\in \germs^n$ by setting $\pert{x}_j = (\Ipert{x}_j, 0_\groupPrime)$. Since the entries of $\epsMat$ are greater than $0_{\groupPrime}$, the lift $\pert{x}$ is feasible for~$\pertLP$.
The rest  follows from  the fact that the map $\first$ is an order-preserving homomorphism, and Proposition~\ref{prop:IpertLP}.
\end{proof}

\subsection{Phase I}\label{subsec:phase_I}

The purpose of this part is to present a method determining whether~$\pertLP$ is feasible, and providing an initial
basis if this is the case. As usual (see~\cite{GroetschelLovaszSchrijver1993}), we use an auxiliary problem, referred to as $\phaseI$, which involves an additional variable $\pert{t}$. This problem arises as the homogenization of the inequalities $\pert{A}^+ \tdot \pert{x} \tplus \pert{b}^+ \geq \pert{A}^- \tdot \pert{x} \tplus \pert{b}^-$ and $\pert{u} \geq \pert{e} \tdot \pert{x}$. 
In order to satisfy Assumptions~\ref{ass:finiteness} and~\ref{ass:non_degeneracy}, we add the perturbations $\pert{d}_i$ and the constraints involving the lower bounds $\pert{l}_j$. More precisely, we consider the following problem:
\begin{equation*}   \label{pb:phase_I} 
\begin{array}{ll}
\text{\rm  maximize} & \, \pert{t} \ \\
\text{\rm subject to} &%
\begin{gathered}[t]%
\pert{A}^+ \tdot \pert{x} \tplus \pert{b}^+ \tdot \pert{t} \tplus \pert{d} \geq \pert{A}^- \tdot \pert{x} \tplus \pert{b}^- \tdot \pert{t} \\
 \pert{u} \tdot \pert{t} \geq \pert{e} \tdot \pert{x} \qquad \pert{Id}_n \tdot \pert{x} \geq \pert{l} \qquad \unit_{\germs} \geq \pert{t}\geq \pert{l}_{n+1} \ ,
\end{gathered}
\end{array}
\tag{$\phaseI$}
\end{equation*}
where 
$\pert{l}_{n+1}$ is chosen of the form $(\Ipert{l}_{n+1}, 0_\groupPrime)$, with $\Ipert{l}_j \ll \Ipert{l}_{n+1} \ll \Ipert{d}_i$ for all $i \in [m]$ and $j \in [n]$. Provided the latter condition on $\Ipert{l}_{n+1}$, Problem $\phaseI$ is trivially feasible, and we even know an initial basis.
\begin{lemma}\label{lemma:basic_point_phase_I}
The vector $(\pert{x}, \pert{t})$ defined by the equalities $\pert{Id}_n \tdot \pert{x} = \pert{l}$ and $\pert{t} = \pert{l}_{n+1}$ is a tropical basic point of $\phaseI$.
\end{lemma}
\begin{proof}
We only need to prove that this vector is feasible. It obviously satisfies the constraints $\pert{u} \tdot \pert{t} \geq \pert{e} \tdot \pert{x}$ as $\Ipert{l}_{n+1} \gg \Ipert{l}_j$ for all $j$. Besides, if $i \in [m]$, we have $\pert{d}_i \geq \pert{A}^-_i \tdot \pert{x} \tplus \pert{b}^-_i \tdot \pert{t}$ thanks to $\Ipert{d}_i \gg \Ipert{l}_j$ for all $j \in [n+1]$.
\end{proof}

To ensure the non-degeneracy of $\phaseI$, we also require  the coefficients $\Ipert{l}_j$ and $\Ipert{d}_i$ to be in distinct ``layers'', \ie\ if we denote $\Ipert{l}_j = (\alpha_j, \cdot)$ and $\Ipert{d}_i = (\alpha'_i, \cdot)$, then we require the scalars $\alpha_1, \dots, \alpha_{n+1}, \alpha'_1, \dots, \alpha'_m$ to be pairwise distinct.
\begin{lemma}\label{lemma:genericity_phase_I}
Problem $\phaseI$ satisfies Assumptions~\ref{ass:finiteness} and~\ref{ass:non_degeneracy}.
\end{lemma}
\begin{proof}
Assumption~\ref{ass:finiteness} is obviously satisfied, so that we just need to prove that the matrix 
\[
\pert{M} = 
\begin{pmatrix}
\pert{A} & \pert{b} & \pert{d} \\
\tminus \pert{e} & \pert{u} & \zero_\germs \\
\pert{Id}_n & \zero_{\germs} & \tminus \pert{l} \\
\zero_{\germs} & \unit_{\germs} & \tminus \pert{l}_{n+1} \\
\zero_{\germs} & \unit_{\germs} & \unit_{\germs}
\end{pmatrix}
\]
is tropically generic. First observe that the matrix
\[
\begin{pmatrix}
\pert{A} & \pert{b} \\
\tminus \pert{e} & \pert{u} \\
\pert{Id}_n & \zero_{\germs}
\end{pmatrix}
=  
\begin{pmatrix}
\Ipert{A} & \Ipert{b} \\
\tminus \Ipert{e} & \Ipert{u} \\
\Ipert{Id}_n & \zero_{\Igerms}
\end{pmatrix}[E]
\]
is tropically generic as $E$ is sufficiently generic. It follows that the submatrix of $\pert{M}$ obtained by removing the last column is also generic.

Now, let $\pert{M}' \in \germs^{I \times J}$ be a submatrix of $\pert{M}$ involving the last column, and such that $\tper \pert{M}' \neq \zero_\germs$. Let us denote $\pert{M}' = (\pert{W} \ \pert{f})$, where $\pert{f}$ is a subcolumn of the last column of $\pert{M}$. Then 
\[
\tper \pert{M}' = \tsum_{i \in I} \pert{f}_i \tdot \tper \pert{W}_{\hat{i}}  \ ,
\]
where  $\pert{W}_{\hat{i}}$ denotes the matrix obtained from $\pert{W}$ by removing the row indexed by $i$. 
Note that for all $i \in I$, the  entries of $\pert{W}_{\hat{i}}$ are either $\zero_{\germs}$ or of the form $((0, \cdot), \cdot)$, so that $\tper \pert{W}_{\hat{i}}$ is either equal to $\zero_\germs$ or of the form $((0, \cdot), \cdot)$. Besides, the coefficients $\pert{f}_i$ are in distinct ``layers'' by assumption, and so the same holds for the terms $\pert{f}_i \tdot \tper \pert{W}_{\hat{i}}$, with $i \in I$. 
As every $\pert{W}_{\hat{i}}$ is generic, we deduce that $\pert{M}'$ is generic too.
\end{proof}

Lemmas~\ref{lemma:basic_point_phase_I} and~\ref{lemma:genericity_phase_I} ensures that Problem $\phaseI$ can be solved by the simplex algorithm presented in Section~\ref{subsec:generic_instances}. 
The following proposition shows that solving $\phaseI$ will provide us the expected information about Problem~$\pertLP$.
\begin{proposition}\label{prop:phase_I}
Problem~$\pertLP$ is feasible if, and only if, the optimal value of $\phaseI$ is equal to $\unit_\germs$. Furthermore, if $(\pert{x}, \unit_\germs)$ is an optimal basic point of $\phaseI$, then $\pert{x}$ is a basic point of~$\pertLP$.
\end{proposition}
\begin{proof}
If $\pert{x}$ is a feasible element of~$\pertLP$, then $(\pert{x}, \unit_\germs)$ is feasible in $\phaseI$, so that its optimal value is indeed $\unit_\germs$. 
Conversely, if $(\pert{x}, \unit_\germs)$ is an optimal basic point of $\phaseI$, then $\pert{x}$ obviously satisfies
the constraints:
\begin{equation}
\begin{aligned}
\pert{A}^+ \tdot \pert{x} \tplus (\pert{b}^+ \tplus \pert{d}) & \geq \pert{A}^- \tdot \pert{x} \tplus \pert{b}^- \\
\pert{u} & \geq \pert{e} \tdot \pert{x} \\ 
\pert{Id}_n  \tdot \pert{x}& \geq \pert{l} \ .
\label{eq:pertLP_constraints}
\end{aligned}
\end{equation}
Moreover, we know that the vector $(\pert{x}, \unit_\germs)$ activates $n+1$ inequalities among the ones defining $\phaseI$. It follows that precisely $n$ inequalities are activated by $\pert{x}$ in~\eqref{eq:pertLP_constraints}.
\end{proof}

The results of this section provide the following theorem.
\begin{theorem}[Tropical simplex method for arbitrary instances] \label{thm:simplex_arbitrary}
An arbitrary tropical linear program $\tropLP(A,b,c)$ on $\trop[\R]$ is solved by the following algorithm:
  \begin{itemize}
\item solve the tropical linear program $\phaseI$ on $ \germs =\trop[\R^{n+3}]$ with the tropical simplex method, starting from the initial basic point defined by $\pert{Id}_n  \tdot \pert{x} = \pert{l}$ and $\pert{t}= \pert{l}_{n+1}$;
\item if the optimal basic point $ (\pert{x}, \pert{t})$ of $\phaseI$ satisfies $ \pert{t} < \unit_{\germs}$, then $\tropLP(A,b,c)$ is infeasible;
\item otherwise, solve the tropical linear program $\pertLP$ with the tropical simplex me\-thod, starting from the initial basic point $\pert{x}$;
\item the optimal basic point $\pert{x}_*$ of $\pertLP$ yields an optimal basic point $\first(\second(\pert{x}_*))$ of $\tropLP(A,b,c)$
\end{itemize}

\end{theorem}

\section{Combinatorial  pivoting rules applied to  mean payoff games}
\label{sec:strongly-polyn-pivot}

We now consider a linear program $\puiseuxLP(\A, \b ,\cc)$ in which the
entries of the matrices $\A,\b,\cc$ belong to an arbitrary real closed
field $K$. We still use the notation $\puiseuxLP(\A, \b ,\cc)$ and $\puiseuxP(\A,\b)$ in this setting. Here we are interested in the two cases where either $K=\R$ or $K=\hahn{G}$, and $G$ is a divisible totally ordered group.  Tarski's principle states that the first-order theory of real closed field is model-complete  (see~\cite{tarski1951decision,seidenberg1954new}), and we will use this to apply results  about the classical simplex method over the reals to linear programming over $\hahn{G}$. 

We begin with the definition of classical counterparts of Assumptions~\ref{ass:finiteness} and~\ref{ass:non_degeneracy}. 
\begin{primedassumption}\label{ass:finiteness2}
The polyhedron $\puiseuxP(\A, \b)$ is bounded, and included in the open orthant given by $\x > 0$.
\end{primedassumption}

\begin{primedassumption}\label{ass:non_degeneracy2}
The matrix 
$\bigl(
\begin{smallmatrix}
\A & \b \\
\transpose{\cc} & 0 
\end{smallmatrix}
\bigr)$ is generic.
\end{primedassumption}

We say, for brevity, that a linear program $\puiseuxLP(\A, \b ,\cc)$ is \emph{non-degenerate} if Assumptions~\ref{ass:finiteness2} and~\ref{ass:non_degeneracy2} are satisfied. We point out that, in this case, the inequalities $\x_j \geq 0$ of the defining system of $\puiseuxP(\A, \b)$ are never activated. As a result, a basis $I$ is necessarily a subset of $[m]$.

Given a linear program $\puiseuxLP(\A,\b, \cc)$ and an initial basis $I_1$, a \emph{run} of the simplex method consists of a finite sequence of bases $I_1, \dots , I_N \subset [m]$, where the last basis yields an optimal basic point. For any $k \leq N-1$, the basis $I_{k+1}$ is of the form $I_{k+1} = I_k \setminus \{\ileaving_k \} \cup \{ \ient_k\}$. 
Under Assumptions~\ref{ass:finiteness2} and ~\ref{ass:non_degeneracy2}, the entering index $ \ient_k$ (and thus the basis $I_{k+1}$), is entirely determined by $\ileaving_k$, $I_k$ and the parameters $\A$ and $ \b$. 

The leaving index $\ileaving_k$ is chosen by a function $\strategy_k$ which takes as input $( I_1, \dots, I_k )$, the \emph{history} up to time $k$,
and the parameters $\A, \b, \cc$.  The family of functions $\strategy = (\strategy_k)_k$ forms a \emph{pivoting strategy}.  We say that the strategy
$\strategy$ is \emph{combinatorial} if at every step $k$, the leaving index $\ileaving_k$ is a function only of the history $( I_1, \dots, I_k )$ and
of the collection of the signs of all the minors of the matrix $\M=\bigl(\begin{smallmatrix} \A & \b \\ \transpose{\cc} & 0\end{smallmatrix}\bigr)$.

Formally, let us denote by \MinorOracle\ the oracle which takes as input a pair $(I,J)$ of non-empty subsets $I\subset [m+1]$ and $J\subset [n+1]$,
having the same cardinality, and returns the sign of the determinant of the $I\times J$-submatrix of $\M$. Then a pivoting strategy is combinatorial
if each function $\strategy_k$ takes as input the history of bases $( I_1, \dots, I_k )$, and is allowed to call the oracle
\MinorOracle{}. 
For instance, Bland's rule, as well as certain lookahead rules exploring a bounded neighborhood of the current basic
point in the graph of the polyhedron, are combinatorial.

Several remarks are in order.  First, the signs of the minors include the orientations of all the edges.  Second, as the number of inequalities, $m$,
and the number of variables, $n$, are not fixed there are super-polynomially many minors.  This means any polynomial time algorithm is clearly
restricted to ``reading'' at most a polynomially bounded number of these signs.  Third, if the (deterministic) pivoting strategy is fixed then the
subsequent bases $I_2, \dots, I_k$ only depend on the initial basis $I_1$.  Yet, we find it more convenient to formally assume that the pivoting may
depend on the entire history.

Any combinatorial pivoting strategy $\strategy$ which applies to non-degenerate instances of classical linear programs can be tropicalized, meaning that it canonically determines a pivoting strategy which applies to non-degenerate instances of tropical linear programs. 
Indeed, the {\em tropicalized strategy}, denoted by $\strategytrop$, is given by the family of functions $\strategy_k$, up to the replacement of the oracle 
$\Omega$ (the sign of a minor) by its tropical analogue.
Note that the signs of the tropical minors of $\Bigl(\begin{smallmatrix}
A & b \\
\transpose{c} & \zero_{\trop}
\end{smallmatrix}\Bigr)$ are
well defined, as this matrix is tropically generic for non-degenerate tropical linear programs (see Section~\ref{subsec:trop_semirings}). 

We denote by $N_{K}(n,m,\strategy)$ the maximal length of a run of the simplex algorithm, equipped with the combinatorial pivoting strategy $\strategy$, for classical linear programs with $n$ variables and $m$ constraints satisfying Assumptions~\ref{ass:finiteness2} and~\ref{ass:non_degeneracy2}, with coefficients in an arbitrary real closed field $K$. 
Similarly, we denote by $N_{\trop}(n,m,\strategytrop)$ the maximal length of a run of the tropical simplex algorithm, equipped with the tropical strategy $\strategytrop$, for a tropical linear program satisfying Assumptions~\ref{ass:finiteness} and~\ref{ass:non_degeneracy}, with coefficients in the semiring $\trop=\trop[G]$.  Recall that $\hahn{\group}$ is real closed if and only if the group $\group$ is divisible \cite[(6.11) p.~143]{ribenboim}.

\begin{proposition}\label{prop:iter_nb}
Assume that $G$ is a divisible totally ordered group.
Then
$
N_{\trop}(n,m,\strategytrop) \leq N_{\R}(n,m,\strategy) 
$.
\end{proposition}

\begin{proof}
  Consider positive integers $n$, $m$, $N$ and a combinatorial pivoting strategy~$\strategy$.
We claim that the statement   ``for any  matrices $\A \in \K^{m \times n}$, $\b \in \K^m$, $\cc \in \K^n$, if the  linear program $\puiseuxLP(\A, \b ,\cc)$ is non-degenerate, then, starting from any initial basis, the simplex algorithm equipped with the combinatorial pivoting strategy $\strategy$ solves $\puiseuxLP(\A, \b ,\cc)$  in at most $N$ iterations'' is a first-order sentence.
Since $\R$ and $\hahn{G}$ are real-closed fields, the equality  $N_{\R}(n,m,\strategy) = N_{\hahn{G}}(n,m,\strategy)$ follows from Tarski's Principle. Finally,  Theorem~\ref{th:tropical_simplex}  yields the inequality $ N_{\trop}(n,m,\strategytrop) \leq  N_{\hahn{G}}(n,m,\strategy)$.

We now prove our claim.
Let $S = (s_{I,J})$ be a vector with entries in $\{-1, 0, +1\}$ indexed by all the pairs $(I, J) \subset [p] \times [q]$ such that $|I|=|J|\geq 1$. 
Given such a vector and a matrix $\M$ of size $p \times q$, we define the formula $F_S(\M)$ which is true if, and only if, the vector $S$ precisely provides the signs of all the minors of the matrix $\M$.
This formula is defined as follows:

\begin{multline*}
F_S(\M) := 
\bigg(\bigwedge_{
s_{I,J} = +1}
\det \M_{I,J} > 0 \bigg)
 \wedge\bigg(
\bigwedge_{
s_{I,J} = 0}
\det \M_{I,J} = 0 \bigg)\\
\wedge
\bigg(\bigwedge_{
s_{I,J} = -1}
\det \M_{I,J}  < 0 \bigg)\ ,
\end{multline*}
where $\M_{I,J}$ denotes the submatrix of $\M$ obtained with rows indexed by $I$ and columns indexed by $J$.
It is clearly a first order formula since any minor $\det \M_{I,J}$ is a polynomial in the entries of $\M$. 

In the sequel, we assume that $S = (s_{I,J})$ denotes the vector of the signs of all minors of the matrix 
$\Bigl(
\begin{smallmatrix}
\A & \b \\
\transpose{\cc} & 0 
\end{smallmatrix}
\Bigr)$,
in particular, $p=m+1$ and $q=n+1$.
The polyhedron $\puiseuxP(\A, \b)$ satisfies Assumption~\ref{ass:finiteness2} if, and only if, the following formula is satisfied:
\begin{multline*}
\mathsf{Bounded}(\A, \b) := \exists (\puiseuxl_j) > 0, (\puiseuxu_j) > 0, \, \forall \x = (\x_j), \, \A \x + \b \geq 0 \implies \\ \forall j \in [n], \, \puiseuxl_j \leq \x_j \leq \puiseuxu_j \ .
\end{multline*}
To determine whether the problem $\puiseuxLP(\A, \b, \cc)$ satisfies Assumption~\ref{ass:non_degeneracy2}, it suffices to examine the sign of the minors of $S$. More precisely, the assumption is fulfilled if, and only if, for all 
$(I, J) \subset [m+1] \times [n+1]$ such that $|I| = |J|$, either $s_{IJ} \neq 0$ or for all bijections $\sigma$ from $I$ to $J$, there exists $i \in I$ such that $s_{\{i\}\{\sigma(i)\}} = 0$. This can be written as follows:
\[
\mathsf{Non\_degenerate}_S := \bigwedge_{\displaystyle
\substack{
I, J \, \text{s.t.}\, s_{I,J} = 0\\
\exists \sigma \in \Sym(I,J), \ \forall i \in I, \ s_{\{i\}\{\sigma(i)\}} \neq 0}}
\text{false} \ .
\]
Indeed, the latter formula is true if, and only if, there is no minor falsifying the assumption.

Verifying that a set $I \subset [m]$ corresponds to a basis of the polyhedron $\puiseuxP(\A,\b)$ can be made by examining the signs of some specific minors. Indeed, $I \subset [m]$ is a basis if, and only if, the following two properties hold:
\begin{itemize}
\item $|I| = n$, and the submatrix $A_I$ 
is invertible, which can be expressed as $s_{I, [n]} \neq 0$.
\item the corresponding basic point $\x$ satisfies the constraints $\A \x + \b \geq 0$, $\x \geq 0$. Note that the entries of $\x$ are given by Cramer determinants. 
Then, given $k \in [m] \setminus I$,
\[
\A_k \x + \b_k = (-1)^n \det\begin{pmatrix}
\A_I & \b_I \\
\A_k & \b_k
\end{pmatrix} / \det \A_I  \ .
\]
It follows that the sign of 
$\A_k \x + \b_k$ can be expressed using $s_{I \cup \{k\}, [n+1]}$ and $s_{I, [n]}$. More precisely, if we denote $I = \{i_1, \dots, i_n\}$ with $i_1 < \dots < i_n$, it can be verified that 
$\A_k \x + \b_k \geq 0$ if, and only if, $(-1)^{l} s_{I, [n]} s_{I \cup \{k\}, [n+1]} \geq 0$ where $l$ is the unique integer such that $i_l < k < i_{l+1}$. Here, if $k < i_1$, we let $l = 0$; similarly, if $k > i_n$, we let $l = n$. 
 \end{itemize}

Given a sequence $I_1, \dots, I_k$ of bases, the pivoting rule switches to the basis $I_{k+1}$ if, and only if, $I_{k+1} \neq I_k$ and $I_{k+1} \supset I_k \setminus \{\strategy_k(I_1, \dots, I_k, S)\}$. Indeed, under the non-degeneracy assumption, we know that the set $K = I_k \setminus \{\strategy_k(I_1, \dots, I_k, S)\}$ corresponds to an edge of the polyhedron $\puiseuxP(\A, \b)$. Consequently, there are only two bases $I$ and $I'$ corresponding basic points incident to the edge.

Finally, to detect whether a basis $I$ corresponds to an optimal basic point, it suffices to check the signs of the reduced costs, which can be expressed again using minors. Thus, the optimality of $I$ can be verified by examining the entries of $S$.

Given $L \geq 1$ and a subset $I_1 \subset [m]$, the following formula $F_{L, I_1}(\A, \b, \cc)$ expresses the fact that if the problem $\puiseuxLP(\A, \b, \cc)$ satisfies the non-degeneracy conditions, then there exists a sequence of precisely $L$ bases starting from the initial basis $I_1$ such that the last one is optimal:
\begin{multline*}
F_{L, I_1}(\A, \b, \cc) := 
\mathsf{Bounded}(\A, \b) \implies \\
\bigvee_{S \in \{-1, 0, 1\}^{\binom{m+n+2}{n+1} -1}} \biggl(F_S
\Bigl(
\begin{smallmatrix}
\A & \b \\
\transpose{\cc} & 0
\end{smallmatrix}
\Bigr) 
\wedge \Bigl(\mathsf{Non\_degenerate}_S \implies \bigl(\bigvee_{(I_1, \dots, I_{L}) \, \text{s.t.} \, (*)} \text{true}\bigr)\Bigr)\biggr) \ . \footnotemark
\end{multline*} 
\footnotetext{Observe that $S$ is a vector of size $\sum_{k = 1}^{n+1} \binom{m+1}{k} \binom{n+1}{k}$, which is equal to $\binom{m+n+2}{n+1}$ by Vandermonde identity.}%
Here, $(*)$ stands for the fact that each $I_k$ is a basis, for all $k \leq L-1$, $I_{k+1}$ is the basis provided by the pivoting rule $\strategy_k(I_1, \dots, I_k, S)$, and $I_{L}$ is an optimal basis. These conditions are entirely expressed in terms of the $s_{I,J}$. Note that there may not be any such sequence of bases, in which case the disjunction $\vee_{(I_1, \dots, I_{L}) \, \text{s.t.} \, (*)} \text{true}$ is false, as expected.

The statement of our claim can now be written as following first-order sentence:
\[
\forall \A, \b, \cc, \, \bigwedge_{I_1 \subset [m]} \bigvee_{L \leq N} F_{L, I_1} (\A, \b, \cc) \ . 
\]
\strut\hfill
\end{proof}

In the sequel, we say that a function can be defined in the \emph{arithmetic model of computation with oracle} when it can be implemented using the arithmetic operations $+$, $-$, $\times$, $/$ over $\R$, and calls to the oracle $\Omega$.
\begin{theorem} \label{th:main}
Let $\strategy$ be a combinatorial pivoting strategy such that the following two conditions are satisfied:
\begin{itemize}
\item each function $\strategy_k$ can be defined in the arithmetic model of computation with oracle,
\item the number of arithmetic operations and calls to the oracle, and the space complexity of every $\strategy_k$ are polynomially bounded by $n$, $m$ and $k$.

\end{itemize}
Suppose that the classical simplex algorithm equipped with $\strategy$ is strongly polynomial on all non-degenerate linear programs over $\R$. Then all tropical linear programs over the max-plus semiring $\trop[\R]$ can be solved by a strongly polynomial algorithm.

\end{theorem}

\begin{proof}
Let $m, n \geq 1$. First observe that $N_\R(n, m, \strategy)$ is polynomially bounded in $n$ and $m$. Indeed, the simplex algorithm equipped with $\strategy$ uses at least one arithmetic operation at each pivoting step. As a consequence, the number of arithmetic operations is greater that $N_\R(n, m, \strategy)$ in the worst case. 

Let $\tropLP(A, b, c)$ be a tropical linear program, with $A \in \trop[\R]^{m \times n}$, $b \in \trop[\R]^m$ and $c \in \trop[\R]^n$. We construct Problems $\pertLP$ and $\phaseI$ as in Section~\ref{sec:handl-degen-line}, by choosing the additional coefficients $\Ipert{d}$ and $\Ipert{l}$ as follows:
\begin{equation}
  \begin{aligned}
   & \Ipert{d}_i := (-i, 0_{\groupPrime})\  &  \text{for all} \ i \in [m] \\
   & \Ipert{l}_j := (-(j+m+1), 0_{\groupPrime}) &\text{for all}  \ j \in [n] \\
   & \Ipert{l}_{n+1} := (-(m+1), 0_{\groupPrime}) &
  \end{aligned}
\label{eq:extra_coeffs}
\end{equation}
We recall that $\groupPrime = \R^{n+1}$, and $\epsMat \in \groupPrime^{(m+n+2) \times (n+1)}$ is given by Lemma~\ref{lemma:perturbation}. In this way, all the assumptions on the $\Ipert{d}$ and $\Ipert{l}$ of Sections~\ref{subsec:general_algorithm} and~\ref{subsec:phase_I} are satisfied. Moreover, the tropical operations $\pert{v} \tplus \pert{w}$ and $\pert{v} \ttimes \pert{w}$ over the semiring $\trop[\R^2 \times \groupPrime]$ can be computed in a number of arithmetic operations (over $\R$) polynomially bounded by $n$, and their space complexity is bounded by a polynomial in the sizes of $\pert{v}$ and $\pert{w}$.

By Theorem~\ref{thm:simplex_arbitrary}, successively applying the tropical simplex algorithm equipped with $\strategytrop$ on Problems $\phaseI$ and $\pertLP$ solves $\tropLP(A, b, c)$. By Proposition~\ref{prop:iter_nb}, the number of iterations is bounded by $N_{\R}(n+1, m+n+3, \strategy)$ and $N_{\R}(n,m+n+1, \strategy)$, respectively. For the two algorithms, at the iteration $k$, the number of arithmetic operations needed to compute the tropical reduced costs and to pivot to the next basis (if needed) is polynomially bounded by $n$ and $m$ due to Theorem~\ref{th:tropical_simplex}. Moreover, since computing the sign of a tropical minor can be done in $O(n^3)$ arithmetic operations over $\trop[\R^2 \times \groupPrime]$, the number of arithmetic operations made by the function $\strategy_k$ is bounded by a polynomial in $n$, $m$ and $k$. Subsequently, the number of arithmetic operations performed by each algorithm is polynomially bounded by $n$ and $m$. 

It remains to show the space complexity is polynomially bounded by the size of the inputs $A$, $b$, and $c$. As discussed above, the space complexity of each elementary operation over $\trop[\R^2 \times \groupPrime]$ is polynomially bounded by the size of their inputs. Similarly, the space complexity of each call to the pivoting rule $\strategy_k$ is polynomially bounded by $n$, $m$, and $k$. As there is a polynomial number in $n$ and $m$ of such operations, we deduce that in total, the space complexity of the algorithms is polynomially bounded by $n$, $m$, and the size of the inputs $\pert{A}_{ij}$, $\pert{b}_i$, $\pert{c}_j$, $\pert{d}_i$, $\pert{u}$, and $\pert{l}_j$. Given the choice made in~\eqref{eq:extra_coeffs}, 
this can be bounded by a polynomial in $n$, $m$, and the size of the entries $A_{ij}$, $b_i$ and $c_j$. 

We conclude that $\tropLP(A, b, c)$ can be solved by strongly polynomial algorithm.
\end{proof}

The following result is an immediate corollary of Theorems~\ref{thm:MPG_to_tropLP} and~\ref{th:main}:
\begin{corollary}\label{cor:main}
Let $\strategy$ be a combinatorial pivoting strategy as in Theorem~\ref{th:main}. Suppose that the classical simplex algorithm equipped with $\strategy$ is strongly polynomial on all non-degenerate linear programs over $\R$. Then mean payoff games can be solved by a strongly polynomial algorithm.
\end{corollary}

\section{Concluding remarks}

It is natural to ask whether Theorem~\ref{th:main} can be extended to more general classes of pivoting strategies. This
leads to the important question of characterizing the classical (deterministic or randomized) pivoting strategies which
admit tropical analogues that can be implemented efficiently. Its study requires a number of very different technical
tools, and so, we leave this for a follow up work.

Also, our perturbation scheme (Section~\ref{sec:handl-degen-line}) might have applications in tropical geometry. In
particular, it is worthwhile to compare this perturbation to the concept of stable intersection.

\appendix

\end{document}